\documentclass{amsart}

\title{Signature characters of highest-weight representations of $U_{q}(\mathfrak{gl}_{n})$}
\author{Vidya Venkateswaran}
\address{Department of Mathematics, MIT, Cambridge, MA 02139}
\email{\href{mailto:vidyav@math.mit.edu}{vidyav@math.mit.edu}}
\thanks{Research supported by NSF Mathematical Sciences Postdoctoral Research Fellowship DMS-1204900}
\subjclass[2000]{}
\keywords{}
\usepackage{hyperref}
\usepackage[svgnames]{xcolor}
\hypersetup{colorlinks,breaklinks,
            linkcolor=DarkBlue,urlcolor=DarkBlue,
            anchorcolor=DarkBlue,citecolor=DarkBlue}
\usepackage{url}
\usepackage[verbose,letterpaper,tmargin=1in,bmargin=1in,lmargin=1.5in,rmargin=1.5in]{geometry}
\usepackage{euscript}
\usepackage[verbose,letterpaper,tmargin=1in,bmargin=1in,lmargin=1.5in,rmargin=1.5in]{geometry}

\usepackage{amssymb}
\usepackage{amsmath,amsthm}
\usepackage{xypic}
\usepackage{verbatim}
\usepackage{mathdots}

\newtheorem{theorem}{Theorem}[section]

\newtheorem{lemma}{Lemma}[section]
\newtheorem{definition}{Definition}
\newtheorem{proposition}{Proposition}[section]
\newtheorem{corollary}{Corollary}[section]
\newtheorem*{remark}{Remark}

\begin{document}

\begin{abstract}
We consider $U_{q}(\mathfrak{gl}_{n})$, the quantum group of type $A$ for $|q| = 1$, $q$ generic.  We provide formulas for signature characters of irreducible finite-dimensional highest weight modules and Verma modules.  In both cases, the technique involves combinatorics of the Gelfand-Tsetlin bases.  As an application, we obtain information about unitarity of finite-dimensional irreducible representations for arbitrary $q$: we classify the continuous spectrum of the unitarity locus.  We also recover some known results in the classical limit $q \rightarrow 1$ that were obtained by different means.  Finally, we provide several explicit examples of signature characters.

\end{abstract}

\maketitle

\section{Introduction}

Let $G$ be a group and $V$ an irreducible complex representation of $G$.  Determining (1) if $V$ admits a non-degenerate invariant Hermitian form and (2) whether this form is positive-definite (i.e., the representation $V$ is \textit{unitary}) is an important problem in representation theory.  The complete classification of unitary representations remains open in many cases (for example, noncompact Lie groups).  

A refinement of unitarity may be found in the context of signature characters.  More precisely, suppose $F$ is a finite-dimensional vector space equipped with a non-degenerate invariant Hermitian form $\langle \cdot, \cdot \rangle$ and an orthogonal basis $B = \{e_{i}\}$ with respect to the form.  We define the signature of $F$ to be
\begin{equation*}
s(F) = \sum_{e_{i} \in B} \text{sgn}(\langle e_{i}, e_{i} \rangle),
\end{equation*}
that is, $s(F)$ is the number of basis elements with positive norm minus the number of basis elements with negative norm (this does not depend on the choice of orthogonal basis).  Now let $V$ be a (possibly infinite-dimensional) irreducible representation, with invariant Hermitian form $\langle \cdot, \cdot \rangle$, orthogonal basis $B$, and weight space decomposition $V = \oplus V_{\mu}$ with $V_{\mu}$ finite-dimensional.  Suppose also that $B$ is compatible with the weight space decomposition.  Then the signature character is
\begin{equation*}
ch_{s}(V) = \sum_{\mu} s(V_{\mu}) e^{\mu},
\end{equation*}
where the sum is over degrees $\mu$.  Note that if $V$ is unitary, $ch_{s}(V)$ is the usual character with respect to the grading.  Thus, the signature character encodes information about failure of unitarity.

In a previous work \cite{V2}, we have computed signature characters for rational Cherednik algebras and Hecke algebras of type $A$.  In \cite{ESG} and \cite{S}, classifications of unitary representations were provided for these algebras, respectively.  In this paper, we consider the case of $U_{q}(\mathfrak{gl}_n)$, the quantum group of type $A$ with $|q|=1$ and $q$ sufficiently generic (so, in particular, it is not a root of unity such that representations we study are reducible).  We use the Gelfand-Tsetlin basis of (1) finite-dimensional, irreducible highest weight modules and (2) Verma modules, along with a combinatorial algorithm, to compute the signature of the norm of an arbitary basis element.  This allows us to compute the signature character (for both the finite and infinite-dimensional representations) in terms of $q$-binomial coefficients.  As a byproduct of our method, the count $s(V_{\mu})$ for a particular weight space can be determined in terms of data coming from the appropriate Gelfand-Tsetlin patterns.

We use our results to obtain information about when the finite-dimensional highest weight representation is unitary for general $q$: namely, for $n \geq 3$, the continuous spectrum of the unitarity locus (in the terminology of \cite{ESG}) is an arc around $q=1$ (with endpoints determined by indexing weight $\lambda$).  As a consequence, in the limit $q \rightarrow 1$ the form on finite-dimensional representations is positive-definite, agreeing with the fact that these representations of $\mathfrak{gl}_n$ are unitary.  For the Verma modules, if $\lambda$ is a negative weight, we recover Wallach's formula \cite{NW} in the limit $q \rightarrow 1$.  We also provide a signature character formula in the limit $q \rightarrow 1$ for arbitrary $\lambda$ for the Verma modules.  Finally, we provide explicit formulas for signature characters for small values of $n$, and particular representations.  

We mention that Wai-Ling Yee \cite{WY, WY2} studied this problem in the classical limit $q \rightarrow 1$ for any Lie algebra.  She used a wall-crossing technique (starting with Wallach's formula \cite{NW} for $\lambda$ in a particular region) to compute formulas for the signature character.  However, our method is different from hers, since our approach relies on the combinatorics of the Gelfand-Tsetlin basis.  

The outline of the paper is as follows.  In the first section, we set up some notation and recall some preliminary facts about representation theory of $U_{q}(\mathfrak{gl}_{n})$.  In the second section, we compute signature characters for the finite-dimensional representations.  In the third section, we compute signature characters for the infinite-dimensional representations.  In both of these sections, we include the applications mentioned above (unitarity and analysis of the classical $q \rightarrow 1$ limit).  Finally, in the last section, we calculate signature characters in some particular cases.

\medskip
\noindent\textbf{Acknowledgements.} The author would like to thank Pavel Etingof and Eric Rains for many helpful discussions and comments.

\section{Preliminaries}

We set up some notation that will be used throughout the article and recall some standard facts and results from the literature; we will follow \cite{NT} and \cite{M}.

Let $q = e^{\pi i s}$ for $0<s<2$ be on the unit circle and for any $k \in \mathbb{Z}$
\begin{equation*}
[k] = [k]_{q}= \frac{q^{k} - q^{-k}}{q-q^{-1}} = \frac{\sin(\pi sk)}{\sin(\pi s)} \in \mathbb{R},
\end{equation*}
and $\{ \cdot \}: \mathbb{R} \setminus 0 \rightarrow \{ \pm 1 \}$ be the sign map.  Note the symmetries $[-k] = (-1) [k]$ and $[k]_{q} = [k]_{q^{-1}}$.  Since norm factors only involve $[k]$, we may restrict to $0< s < 1$, in particular $\sin(\pi s)$ is positive.  We also note that $\{[k]\} = (-1)^{\lfloor ks \rfloor}$, and $\lim_{q \rightarrow 1} [k] = k$.

For $\lambda = (\lambda_{1}, \dots, \lambda_{n})$ with $\lambda_{i} - \lambda_{i+1} \in \mathbb{Z}_{\geq 0}$, let $L(\lambda)$ denote the finite-dimensional $U_{q}(\mathfrak{gl_{n}})$-module  with highest weight $\lambda$.  Recall that this module has a basis $\{ \zeta_{\Lambda} \}$ indexed by Gelfand-Tsetlin patterns $\Lambda = (\lambda_{i,j})$ for $1 \leq j \leq i \leq n$ with $\lambda_{n,i} = \lambda_{i}$ for $1 \leq i \leq n$, which comes from the multiplicity-one decomposition associated with the chain of subalgebras
\begin{equation*}
U_{q}(\mathfrak{gl}_{1}) \subset U_{q}(\mathfrak{gl}_{2}) \subset \cdots \subset U_{q}(\mathfrak{gl}_{n}).
\end{equation*}
Recall that Gelfand-Tsetlin patterns satisfy the \textit{interlacing condition}
\begin{equation} \label{interl}
\lambda_{m,i} \geq \lambda_{m-1,i} \geq \lambda_{m, i+1}
\end{equation}
for all $m,i$.  
These arrays may be visualized as follows
\begin{equation*}
\begin{array}{ccccccccc}
\lambda_{1} &  & \lambda_{2} &  &\cdots &  & \lambda_{n-1} &  &
\lambda_{n} \\
& \lambda_{n-1,1} &  & \lambda_{n-1,2} &   \cdots & \lambda_{n-1,n-2}&  & \lambda_{n-1,n-1} &  \\
&  & \ddots &  & \cdots &  & \iddots &  &  \\
&  &  & \lambda_{2,1} &  & \lambda_{2,2} &  &  &  \\
&  &  &  & \lambda_{1,1} &  &  &  &
\end{array}%
\end{equation*}
The highest weight element is $\zeta_{\Lambda_{0}}$, where $\Lambda_{0}$ is the pattern with $\lambda_{mi} = \lambda_{i}$ for all $m,i$.  We will write $GT(\lambda)$ to denote the Gelfand-Tsetlin basis of $L(\lambda)$.

For $\lambda = (\lambda_{1}, \dots, \lambda_{n})$ with $\lambda_{i} - \lambda_{i+1} \in \mathbb{R}$ (i.e., generic differences), let $M(\lambda)$ denote the Verma module with highest weight $\lambda$.  Recall that this module has a basis $\{\zeta_{\Lambda}\}$ indexed by arrays $\Lambda = (\lambda_{i,j})$ for $1 \leq j \leq i \leq n$ with $\lambda_{n,i} = \lambda_{i}$ for $1 \leq i \leq n$.  Such arrays must satisfy the condition
\begin{equation*}
\lambda_{mi} - \lambda_{m-1,i} \in \mathbb{Z}_{\geq 0}
\end{equation*}
 for all $m,i$ (crucially, the interlacing condition need not hold).

 The highest weight element is $\zeta_{\Lambda_{0}}$ where $\Lambda_{0}$ is the array with $\lambda_{mi} = \lambda_{i}$ for all $m,i$.  We will write $A(\lambda)$ to denote this basis of $M(\lambda)$.
 
For $\zeta_{\Lambda} \in L(\lambda)$ or $M(\lambda)$ we define the weight function by
\begin{equation} \label{wt}
\text{wt}(\zeta_{\Lambda}) = (\lambda_{11}, \lambda_{21} + \lambda_{22} - \lambda_{11}, \dots, \sum_{i=1}^{n-1} \lambda_{n-1,i} - \sum_{i=1}^{n-2} \lambda_{n-2,i}, \sum_{i=1}^{n} \lambda_{i} - \sum_{i=1}^{n-1} \lambda_{n-1,i}) \in \mathbb{R}^{n}.
\end{equation} 
 Note that $\lambda - \text{wt}(\zeta_{\Lambda})$ is a vector in $\Lambda_{r}^{+}$, the positive root lattice.  So $\text{wt}(\zeta_{\Lambda}) = \lambda - \mu$, for $\mu \in \Lambda_{r}^{+}$.
 
We will write $( \cdot, \cdot)$ to denote the standard Hermitian form on $L(\lambda)$ and $M(\lambda)$.  We recall that bases $GT(\lambda)$ and $A(\lambda)$ are orthogonal with respect to this form, and that the operators $e_{m}, f_{m} \in U_{q}(\mathfrak{gl}_{n})$ are adjoint with respect to this form.  We will compute $ch_{s}(L(\lambda))$ and $ch_{s}(M(\lambda))$ by computing the signs of the norms $(\zeta_{\Lambda}, \zeta_{\Lambda})$ for basis elements $\zeta_{\Lambda}$.  We will use a combinatorial algorithm that describes $(\zeta_{\Lambda}, \zeta_{\Lambda})$ in terms of $(\zeta_{\Lambda_{0}}, \zeta_{\Lambda_{0}})$.

We define the coefficients
\begin{equation*}
\nu_{m,i} = i - \lambda_{m,i} - 1
\end{equation*}
and 
\begin{equation*}
\nu_{i} = i - \lambda_{i} -1.
\end{equation*}
We will write $v_{m,i}^{\Lambda}$ to indicate the underlying array $\Lambda$ if it is not clear from context (for example, if we are comparing different arrays and wish to emphasize which array the coefficients are coming from).

We also define the following coefficients, which are expressed in terms of $[k]$:
\begin{equation*}
\beta_{mi \Lambda} = \prod_{j=1}^{i} \frac{[\nu_{mi} - \nu_{m+1,j} + 1]}{[\nu_{mi} - \nu_{j}+1]} \prod_{j=1}^{i-1} \frac{[\nu_{mi}-\nu_{m-1,j}}{[\nu_{mi}-\nu_{j}]}
\end{equation*}
and
\begin{equation*}
\gamma_{m i \Lambda} = \prod_{j=1}^{i} [\nu_{mi} - \nu_{j}] \prod_{j=1}^{i-1} [\nu_{mi} - \nu_{j}-1] \prod_{j=i+1}^{m+1} [\nu_{m+1,j} - \nu_{mi}] \prod_{j=i}^{m-1} [\nu_{m-1,j} - \nu_{mi} + 1]
\end{equation*}
and
\begin{equation*}
\tau_{mi\Lambda} = \prod_{\substack{j=1 \\ j \neq i}}^{m} \frac{1}{[\nu_{mi} - \nu_{mj}]}.
\end{equation*}

We also write, for $X \in \mathbb{Z}_{> 0}$
\begin{equation*}
[X]! = [X] [ X-1] \cdots [2] [1]
\end{equation*}
and for $X \in \mathbb{R}, k \in \mathbb{Z}_{+}$
\begin{equation*}
[X]_{k}! = [X] [ X-1] \cdots [X-k]
\end{equation*}
and
\begin{equation*}
(X)_{k}! = X(X-1) \cdots (X-k) = \lim_{q \rightarrow 1} [X]_{k}!.
\end{equation*}
We note that (assuming $X \in \mathbb{R} \setminus \mathbb{Z}$ and $k \in \mathbb{Z}_{+}$)
\begin{multline} \label{sgncomp}
\{ (X)_{k}! \} = \begin{cases} 1 ,& \text{if } X-k > 0 \\
(-1)^{k+1} ,& \text{if } X < 0 \\
(-1)^{\lfloor X-k \rfloor} ,& \text{if } X > 0 \text{ and } X-k<0
\end{cases} \\
= (-1)^{\min \{0, \lfloor X \rfloor + 1\}} (-1)^{\min \{0, \lfloor X-k  \rfloor \}}.
\end{multline}

We will need the following result which describes how the quantum group operators act on basis elements in terms of Gelfand-Tsetlin patterns.

\begin{theorem}  [\cite{NT}] \label{NazTar} Let $\zeta_{\Lambda} \in GT(\lambda)$ and $m<n$.  The operators $e_{m}$ and $f_{m}$ (for $m=1, \dots, n-1$) act on the module $L(\lambda)$ as follows
\begin{equation*}
e_{m} \cdot \zeta_{\Lambda} = \sum_{\Lambda^{+}} \gamma_{m i \Lambda} \tau_{m i \Lambda} \zeta_{\Lambda^{+}}
\end{equation*}
and
\begin{equation*}
f_{m} \cdot \zeta_{\Lambda} = \sum_{\Lambda^{-}} \beta_{m i \Lambda} \tau_{m i \Lambda} \zeta_{\Lambda^{-}},
\end{equation*}
where $\Lambda^{+}$ and $\Lambda^{-}$ are Gelfand-Tsetlin patterns obtained from $\Lambda$ by increasing and decreasing (resp.) the $(m,i)$-entry by $1$ (where the sum is over all $i \leq m$).  
\end{theorem}

\section{Signature characters of finite-dimensional modules}

Let $L(\lambda)$ be the irreducible finite-dimensional representation with highest weight $\lambda$.  Let $\Lambda$ be a Gelfand-Tsetlin pattern with first row equal to $\lambda$; recall that it satisfies the interlacing condition (\ref{interl}).  We will compute the sign $\{ ( \zeta_{\Lambda}, \zeta_{\Lambda} ) \}$ through a series of lemmas that will illustrate the combinatorial technique mentioned in the Introduction.

\begin{lemma}
Let $\Lambda = (\lambda_{i,j})$ be a Gelfand-Tsetlin pattern with first row $\lambda = (\lambda_{1}, \dots, \lambda_{n})$.  Let $X$ be the array with entry $(m,i)$ of $\Lambda$ increased by $1$ (for $1 \leq i \leq m < n$), and all other entries remain the same.  We have 
\begin{equation*}
(\zeta_{X}, \zeta_{X}) = \frac{\beta_{m i X} \tau_{m i X}}{\gamma_{m i \Lambda} \tau_{m i \Lambda}} (\zeta_{\Lambda}, \zeta_{\Lambda}).
\end{equation*}
\end{lemma}
\begin{proof}
Fix $m,i$.  Let $X \in \Lambda^{+}$, then $\Lambda \in X^{-}$ (both by changing the $(m,i)$-entry).  So by Theorem \ref{NazTar} and orthogonality, we have
\begin{equation*}
(\zeta_{X}, e_{m}\zeta_{\Lambda}) = \gamma_{m i \Lambda} \tau_{m i \Lambda} (\zeta_{X}, \zeta_{X}).
\end{equation*}
On the other hand, $(\zeta_{X}, e_{m}\zeta_{\Lambda}) = (f_{m}\zeta_{X}, \zeta_{\Lambda})$ by adjointness, and again by Theorem \ref{NazTar} and orthogonality
\begin{equation*}
(f_{m}\zeta_{X}, \zeta_{\Lambda}) = \beta_{m i X} \tau_{m i X} (\zeta_{\Lambda}, \zeta_{\Lambda}).
\end{equation*}
So we have
\begin{equation*}
(\zeta_{X}, \zeta_{X}) = \frac{\beta_{m i X} \tau_{m i X}}{\gamma_{m i \Lambda} \tau_{m i \Lambda}} (\zeta_{\Lambda}, \zeta_{\Lambda}).
\end{equation*}
\end{proof}

We will now measure the sign change between the two norms.

\begin{lemma}\label{onestep} Let $\Lambda = (\lambda_{i,j})$ be a Gelfand-Tsetlin pattern with first row $\lambda = (\lambda_{1}, \dots, \lambda_{n})$.  Let $X$ be the array with entry $(m,i)$ of $\Lambda$ increased by $1$, and all other entries the same.  We have
\begin{equation*}
\{ (\zeta_{X}, \zeta_{X}) \} = s_{m i \Lambda} \{ (\zeta_{\Lambda}, \zeta_{\Lambda}) \},
\end{equation*}
where $s_{m i \Lambda}$ is equal to
\begin{equation*}
 (-1) \Big \{ \prod_{j=1, j \neq i}^{m}  [\nu_{mi}^{\Lambda} - \nu_{mj}^{\Lambda}][\nu_{mi}^{\Lambda} - \nu_{mj}^{\Lambda}-1] \Big\}  \Big\{ \prod_{j=1}^{m+1} [\nu_{mi}^{\Lambda} - \nu_{m+1,j}^{\Lambda}] \Big\} 
\Big\{ \prod_{j=1}^{m-1} [\nu_{mi}^{\Lambda} - 1 - \nu_{m-1,j}^{\Lambda}] \Big \}. 
\end{equation*}
\end{lemma}

\begin{proof}
By the previous lemma, we have
\begin{multline*}
(\zeta_{X}, \zeta_{X}) \\ = \frac{\beta_{m i X} \tau_{m i X}}{\gamma_{m i \Lambda} \tau_{m i \Lambda}} (\zeta_{\Lambda}, \zeta_{\Lambda}) = \prod_{j=1, j \neq i}^{m} \frac{[\nu_{mi}^{\Lambda} - 1 - \nu_{mj}^{\Lambda}]^{-1}}{[\nu_{mi}^{\Lambda} - \nu_{mj}^{\Lambda}]^{-1}} \prod_{j=1}^{i} \frac{[\nu_{mi}^{\Lambda} - \nu_{m+1,j}^{\Lambda}]}{[\nu_{mi}^{\Lambda} - \nu_{j}]} \prod_{j=1}^{i-1} \frac{[\nu_{mi}^{\Lambda}-1-\nu_{m-1,j}^{\Lambda}]}{[\nu_{mi}^{\Lambda}-1-\nu_{j}]} \\
\times \prod_{j=1}^{i} \frac{1}{[\nu_{mi}^{\Lambda} - \nu_{j}]} \prod_{j=1}^{i-1} \frac{1}{[\nu_{mi}^{\Lambda}-\nu_{j}-1]} \prod_{j=i+1}^{m+1} \frac{1}{[\nu_{m+1,j}^{\Lambda} - \nu_{mi}^{\Lambda}]} \prod_{j=i}^{m-1} \frac{1}{[\nu_{m-1,j}^{\Lambda} - \nu_{mi}^{\Lambda} + 1]} (\zeta_{\Lambda}, \zeta_{\Lambda}),
\end{multline*}
where we have used $\nu_{m,i}^{X} = i - \lambda_{m,i}^{X} - 1 = i - (\lambda_{m,i}^{\Lambda} + 1) - 1 = \nu_{m,i}^{\Lambda} - 1$ and $\lambda_{m,i}^{X}$ denotes the $(m,i)$ entry of $X$.  Using $[k] = (-1)[-k]$ and looking at signs only, we have
\begin{multline*}
\{ (\zeta_{X}, \zeta_{X}) \} = \{ (\zeta_{\Lambda}, \zeta_{\Lambda}) \} \Big\{ \prod_{j=1, j \neq i}^{m}  [\nu_{mi}^{\Lambda} - \nu_{mj}^{\Lambda}][\nu_{mi}^{\Lambda} - \nu_{mj}^{\Lambda}-1] \Big\} (-1)^{m-i+1} \\ \times \Big\{ \prod_{j=1}^{m+1} [\nu_{mi}^{\Lambda} - \nu_{m+1,j}^{\Lambda}] \Big\} 
 (-1)^{m-i} \Big\{ \prod_{j=1}^{m-1} [\nu_{mi}^{\Lambda} - 1 - \nu_{m-1,j}^{\Lambda}] \Big \} \\
=  \{ (\zeta_{\Lambda}, \zeta_{\Lambda}) \} (-1) \Big \{ \prod_{j=1, j \neq i}^{m}  [\nu_{mi}^{\Lambda} - \nu_{mj}^{\Lambda}][\nu_{mi}^{\Lambda} - \nu_{mj}^{\Lambda}-1] \Big\}  \Big\{ \prod_{j=1}^{m+1} [\nu_{mi}^{\Lambda} - \nu_{m+1,j}^{\Lambda}] \Big\} \\
\times \Big\{ \prod_{j=1}^{m-1} [\nu_{mi}^{\Lambda} - 1 - \nu_{m-1,j}^{\Lambda}] \Big \}, 
\end{multline*}
as desired. 
\end{proof}

\begin{lemma} Let $\Lambda = (\lambda_{i,j})$ be a Gelfand-Tsetlin array with first row $\lambda = (\lambda_{1}, \dots, \lambda_{n})$.  Suppose that increasing the $(m,i)$ entry of $\Lambda$ by $k$ (and keeping the other entries fixed) results in a Gelfand-Tsetlin pattern $X$.  We have
\begin{equation*}
\{ (\zeta_{X}, \zeta_{X} \} = s_{m i \Lambda}^{(k)} \{ (\zeta_{\Lambda}, \zeta_{\Lambda}) \},
\end{equation*}
where
\begin{multline*}
s_{m i \Lambda}^{(k)} = \Big \{ \prod_{j=1, j \neq i}^{m} [\nu_{mi}^{\Lambda} - \nu_{mj}^{\Lambda}] [\nu_{mi}^{\Lambda} - k - \nu_{mj}^{\Lambda}] \Big \} (-1)^{k} \\
\times \Big \{  \prod_{j=1}^{m+1} [\nu_{mi}^{\Lambda} - \nu_{m+1,j}^{\Lambda}] [\nu_{mi}^{\Lambda} - \nu_{m+1,j}^{\Lambda} - 1] \cdots [\nu_{mi}^{\Lambda} - \nu_{m+1, j}^{\Lambda} - (k-1)] \Big\} \\
\times \Big \{ \prod_{j=1}^{m-1} [\nu_{mi}^{\Lambda} - \nu_{m-1,j}^{\Lambda}-1] [\nu_{mi}^{\Lambda} - \nu_{m-1,j}^{\Lambda}-2] \cdots [\nu_{mi}^{\Lambda} - \nu_{m-1,j}^{\Lambda} - k] \Big \}.
\end{multline*}
\end{lemma}
\begin{proof}
We iterate the result of the previous lemma $k$-times, and take into account signs, to obtain the formula.  
\end{proof}

\begin{lemma}
Let $1 \leq i \leq m <n$ be fixed.  Let $\Lambda$ be a Gelfand-Tsetlin pattern with first row $\lambda = (\lambda_{1}, \dots, \lambda_{n})$.  Suppose it has entries $\lambda_{i,j}$ such that $\lambda_{l,j} = \lambda_{j}$ for $j<i$ and any $l$ and $\lambda_{l,i} = \lambda_{i}$ for $l>m$.  Let $X$ be the Gelfand-Tsetlin pattern that agrees with $\Lambda$ in all entries except $(m,i)$, where it is equal to $\lambda_{m,i} + (\lambda_{i} - \lambda_{m,i}) = \lambda_{i}$.  We have
\begin{equation*}
\{ (\zeta_{X}, \zeta_{X} ) \} = s_{m i \Lambda}^{(\lambda_{i} - \lambda_{m,i})} \{ (\zeta_{\Lambda}, \zeta_{\Lambda}) \},
\end{equation*}
where
\begin{multline*}
s_{m i \Lambda}^{(\lambda_{i} - \lambda_{m,i})}  = \Big\{ \prod_{i<j \leq m} [\nu_{mi}^{\Lambda} - \nu_{mj}^{\Lambda}][\nu_{i} - \nu_{mj}^{\Lambda}] \Big\}  (-1)^{\nu_{mi}^{\Lambda} - \nu_{i}} \{ [\nu_{mi}^{\Lambda} - \nu_{i}]_{(\lambda_{i} - \lambda_{m,i}-1)}! \} \\
\times \Big \{ \prod_{i < j \leq m+1} [\nu_{mi}^{\Lambda} - \nu_{m+1,j}^{\Lambda}]_{(\lambda_{i} - \lambda_{m,i}-1)}! \Big \} \Big \{ \prod_{i \leq j \leq m-1} [\nu_{mi}^{\Lambda} - \nu_{m-1,j}^{\Lambda} - 1]_{(\lambda_{i} - \lambda_{m,i}-1)}! \Big \}.
\end{multline*}
\end{lemma}
\begin{proof}
We first note that, because of the specified parts of $\Lambda$, we have
\begin{equation*}
\nu_{mj}^{\Lambda} = \nu_{j} \text{ for } j<i
\end{equation*}
\begin{equation*}
\nu_{m+1,i}^{\Lambda} = \nu_{i}
\end{equation*}
\begin{equation*}
\nu_{m-1,j}^{\Lambda} = \nu_{j} \text{ for } j<i
\end{equation*}
\begin{equation*}
\nu_{m+1,j}^{\Lambda} = \nu_{j} \text{ for } j<i.
\end{equation*}
We then use the previous Lemma for this particular choice of $\Lambda$ and cancel terms according to whether they appear an even or odd number of times.
\end{proof}

\begin{theorem}
Let $\Lambda$ be a Gelfand-Tsetlin pattern with first row $\lambda = (\lambda_{1}, \dots, \lambda_{n})$.  Then
\begin{multline*}
\{ (\zeta_{\Lambda}, \zeta_{\Lambda}) \}  = \prod_{1 \leq i \leq m < n} s_{m i \Lambda}^{(\lambda_{i} - \lambda_{m,i})}\\ =\prod_{1 \leq i \leq m < n} \Big\{ \prod_{i<j \leq m} [\nu_{mi}^{\Lambda} - \nu_{mj}^{\Lambda}][\nu_{i} - \nu_{mj}^{\Lambda}] \Big\}  (-1)^{\nu_{mi}^{\Lambda} - \nu_{i}} \{ [\nu_{mi}^{\Lambda} - \nu_{i}]_{(\lambda_{i} - \lambda_{m,i}-1)}! \} \\
\times \Big \{ \prod_{i < j \leq m+1} [\nu_{mi}^{\Lambda} - \nu_{m+1,j}^{\Lambda}]_{(\lambda_{i}-\lambda_{m,i}-1)}! \Big \} \Big \{ \prod_{i \leq j \leq m-1} [\nu_{mi}^{\Lambda} - \nu_{m-1,j}^{\Lambda} - 1]_{(\lambda_{i} - \lambda_{m,i}-1)}! \Big \}. 
\end{multline*}
\end{theorem}
\begin{proof}
Obtained from the previous Lemma, starting with $\Lambda$ and applying the required series of steps in order to produce $\Lambda_{0}$.
\end{proof}

\begin{definition} \label{eltsgn}
For $\zeta_{\Lambda} \in GT(\lambda)$ let
\begin{multline*}
s_{q}(\zeta_{\Lambda}) = \prod_{1 \leq i \leq m < n} \Big\{ \prod_{i<j \leq m} [\nu_{mi} - \nu_{mj}][\nu_{i} - \nu_{mj}] \Big\}  (-1)^{\nu_{mi}- \nu_{i}} \{ [\nu_{mi} - \nu_{i}]_{(\lambda_{i} - \lambda_{m,i}-1)}! \} \\
\times \Big \{ \prod_{i < j \leq m+1} [\nu_{mi} - \nu_{m+1,j}]_{(\lambda_{i}-\lambda_{m,i}-1)}! \Big \} \Big \{ \prod_{i \leq j \leq m-1} [\nu_{mi}- \nu_{m-1,j} - 1]_{(\lambda_{i} - \lambda_{m,i}-1)}! \Big \},
\end{multline*}
which is equal to $\pm 1$, depending on $\Lambda$ and $q$.

\end{definition}

As a result of the previous computation of the sign of the norm for any basis element and Definition \ref{eltsgn}, we obtain a formula for the signature character.

\begin{theorem} \label{finsigchar}
We have
\begin{equation*}
ch_{s}(L(\lambda)) = \sum_{\zeta_{\Lambda} \in GT(\lambda)} s_{q}(\zeta_{\Lambda}) e^{\text{wt}(\zeta_{\Lambda})}.
\end{equation*}
\end{theorem}

Recall that $\text{wt}(\cdot)$ was defined in (\ref{wt}).

We will provide an alternate formula for $s_{q}(\zeta_{\Lambda})$ by cancelling off signs.  This will be easier to work with.

\begin{lemma}
Let $\Lambda$ be a Gelfand-Tsetlin pattern with first row $\lambda = (\lambda_{1}, \dots, \lambda_{n})$.  We have
\begin{multline} \label{finqsig}
s_{q}(\zeta_{\Lambda})= \prod_{1 \leq i \leq m < n} \Big\{ \prod_{i<j \leq m} [\nu_{mj}^{\Lambda} - \nu_{mi}^{\Lambda}][\nu_{mj}^{\Lambda} - \nu_{i}] \Big\} \{ [\nu_{mi}^{\Lambda} - \nu_{i}]_{(\lambda_{i} - \lambda_{m,i}-1)}! \} \\
\times \Big \{ \prod_{i < j \leq m+1} [\nu_{m+1,j}^{\Lambda} - \nu_{mi}^{\Lambda} + (\lambda_{i}-\lambda_{m,i}-1) ]_{(\lambda_{i}-\lambda_{m,i}-1)}! \Big \} \\
\times  \Big \{ \prod_{i \leq j \leq m-1} [\nu_{m-1,j}^{\Lambda} -\nu_{mi}^{\Lambda}  + (\lambda_{i} - \lambda_{m,i})]_{(\lambda_{i} - \lambda_{m,i}-1)}! \Big \} \\
= \prod_{1 \leq i \leq m < n} \Big\{ \prod_{i<j \leq m} [\nu_{mj}^{\Lambda} - \nu_{mi}^{\Lambda}][\nu_{mj}^{\Lambda} - \nu_{i}] \Big\} \{ [\nu_{mi}^{\Lambda} - \nu_{i}]_{(\lambda_{i} - \lambda_{m,i}-1)}! \} \\
\times \Big \{ \prod_{i < j \leq m+1} [\nu_{m+1,j}^{\Lambda} - \nu_{i}-1 ]_{(\lambda_{i}-\lambda_{m,i}-1)}! \Big \} \\
\times  \Big \{ \prod_{i \leq j \leq m-1} [\nu_{m-1,j}^{\Lambda} -\nu_{i} ]_{(\lambda_{i} - \lambda_{m,i}-1)}! \Big \}
\end{multline}
where now all $q$-binomial coefficients $[n]$ appearing above satisfy $n \geq 0$.  
\end{lemma}
\begin{proof}
 We start with the formula for $s_{q}(\zeta_{\Lambda})$ in Definition \ref{eltsgn} and note that $[m]_{q} = (-1)[-m]_{q}$ for all $m \in \mathbb{Z}$.  We use the interlacing condition to determine the sign of terms $m$ appearing in the $q$-binomials of the formula.  For $i<j \leq m$ we have
\begin{equation*}
\nu_{mi}^{\Lambda} - \nu_{mj}^{\Lambda} = [i-\lambda_{mi} - 1]  - [j-\lambda_{mj} - 1] = [i-j] + [\lambda_{mj} - \lambda_{mi}] \leq 0
\end{equation*}
\begin{equation*}
\nu_{i} - \nu_{mj}^{\Lambda} = [i - \lambda_{i} - 1] - [j-\lambda_{mj}-1] = [i-j] + [\lambda_{mj} - \lambda_{i}] \leq 0.
\end{equation*}
We also have
\begin{equation*}
\nu_{mi}^{\Lambda} - \nu_{i} = [i - \lambda_{mi}-1] - [i-\lambda_{i} - 1] = \lambda_{i} - \lambda_{mi} \geq 0.
\end{equation*}
For $i<j \leq m+1$, we have
\begin{equation*}
\nu_{mi}^{\Lambda} - \nu_{m+1,j}^{\Lambda} = [i - \lambda_{mi} -1 ] - [j - \lambda_{m+1,j} - 1] = [i-j] + [\lambda_{m+1,j} - \lambda_{mi}] \leq 0.
\end{equation*}
Finally, for $i \leq j \leq m-1$, we have
\begin{multline*}
\nu_{mi}^{\Lambda} - \nu_{m-1,j}^{\Lambda} - 1 = [i - \lambda_{mi} - 1] - [j-\lambda_{m-1,j} - 1] - 1\\
= [i-j] + [\lambda_{m-1,j} - \lambda_{m,i}] - 1 < 0.
\end{multline*}
Putting the negative signed contributions together yields
\begin{equation*}
(-1)^{\nu_{mi}^{\Lambda}-\nu_{i}} (-1)^{(\lambda_{i} - \lambda_{m,i})(m+1-i)} (-1)^{(\lambda_{i} - \lambda_{m,i})(m-i)} 
= (-1)^{\nu_{mi}^{\Lambda} - \nu_{i} + \lambda_{i} - \lambda_{m,i}} = 1
\end{equation*}
as desired.
\end{proof}

We now use Theorem \ref{finsigchar} to obtain some results about unitarity of $L(\lambda)$ for particular values of $s$ (recall that $q = e^{\pi i s}$).  In particular, for $n=2$, we will give a characterization of such $s$, and for $n \geq 3$ we will show that the only interval is the one around $q=1$.

\begin{lemma} \label{n=2char}
Let $n=2$ and $\lambda = (\lambda_{1}, \lambda_{2})$.  Then $L(\lambda)$ is unitary if and only if the signs
\begin{equation*}
\{ [d] \} = \{ [\lambda_{1} - \lambda_{2} - d + 1] \},
\end{equation*}
for all $1 \leq d \leq \lambda_{1} - \lambda_{2}$.
\end{lemma}
\begin{proof}
We let $\Lambda$ be the Gelfand-Tsetlin pattern with $\lambda_{1}, \lambda_{2}$ in the first row and $\lambda_{11} = \lambda_{1} - d$ in the second row for $0 \leq d \leq \lambda_{1} - \lambda_{2}$.  We use Equation (\ref{finqsig}) to compute
\begin{equation*}
s_{q}(\zeta_{\Lambda}) = \{ [d] [d-1] \cdots [2][1] \} \times \{ [\lambda_{1} - \lambda_{2}] [\lambda_{1} - \lambda_{2}-1] \cdots [\lambda_{1}-\lambda_{2} - d+1] \}.
\end{equation*}
So, if $L(\lambda)$ is unitary if and only if the ratio
\begin{equation*}
\frac{\{ [d] [d-1] \cdots [2][1] \} \times \{ [\lambda_{1} - \lambda_{2}] [\lambda_{1} - \lambda_{2}-1] \cdots [\lambda_{1}-\lambda_{2} - d+1] \}}{\{ [d-1] \cdots [2][1] \} \times \{ [\lambda_{1} - \lambda_{2}] [\lambda_{1} - \lambda_{2}-1] \cdots [\lambda_{1}-\lambda_{2} - d+2] \}} = \{ [d] [\lambda_{1} - \lambda_{2} - d + 1] \}
\end{equation*}
is equal to $1$, which gives the result.
\end{proof}

\begin{lemma}
Let $\lambda = (\lambda_{1}, \lambda_{2})$ and $\lambda' = (\lambda_{1} + 1, \lambda_{2})$.  Then if $L(\lambda)$ and $L(\lambda')$ are both unitary, the signs $\{[1]\}, \{ [2]\}, \cdots \{ [\lambda_{1} - \lambda_{2}]\}, \{[\lambda_{1}-\lambda_{2} + 1] \}$ are all equal to $1$, or equivalently, $s < \frac{1}{\lambda_{1} - \lambda_{2} + 1}$.
\end{lemma}
\begin{proof}
Using the Lemma \ref{n=2char} for $L(\lambda)$ gives that $[1]$ and $[\lambda_{1} - \lambda_{2}]$ have the same sign.  Similarly, applying it to $L(\lambda')$ gives that $[1]$ and $[\lambda_{1} - \lambda_{2} + 1]$ have the same sign, as do $[2]$ and $[\lambda_{1} - \lambda_{2}]$.  By iterating this argument and using $\{ [1] \} = 1$, we get the result.
\end{proof}

\begin{theorem}
Let $n \geq 3$ and $q = e^{\pi i s}$.  Then the module $L(\lambda)$ is unitary if and only if $0 < s < 1/M_{\lambda}$, where
\begin{equation*}
M_{\lambda} = (\lambda_{1} - \lambda_{n}) + (n-2).
\end{equation*}
\end{theorem}
\begin{proof}

First we will show that $L(\lambda)$ is unitary for $0<s< \frac{1}{M_{\lambda}}$, but it is not unitary for $s = \frac{1}{M_{\lambda}} + \epsilon$ for $\epsilon$ small.  Note first the sign computation
\begin{equation*}
\{ [k]_{q} \} = \Big\{\frac{e^{k\pi i s} - e^{-k\pi i s}}{e^{\pi i s} - e^{-\pi i s}} \Big\} = \Big\{ \frac{\sin(k\pi s)}{\sin(\pi s)}\Big \} = \{\sin(k \pi s) \} = (-1)^{\lfloor ks \rfloor},
\end{equation*}
since we are restricting to $0<s<1$ so that $\sin(\pi s)$ is positive.  Let $0< s < \frac{1}{M_{\lambda}}$ and $\zeta_{\Lambda} \in GT(\lambda)$ be arbitrary.  Recall the formula for $s_{q}(\zeta_{\Lambda})$ in Equation (\ref{finqsig}).  Using the interlacing condition, one can check that each term $[k]$ appearing in that formula satisfies 
$0<k \leq \nu_{n} - \nu_{1}-1$, so that
\begin{equation*}
 0< ks  \leq (\nu_{n} - \nu_{1}-1)s < \frac{\nu_{n} - \nu_{1}-1}{M_{\lambda}} = 1.
\end{equation*}
Thus, $\lfloor ks \rfloor = 0$, so that each factor appearing in the formula is positive.  Thus, $s_{q}(\zeta_{\Lambda}) = 1$ for all $\zeta_{\Lambda}$, so the representation $L(\lambda)$ is unitary for $s$ in the interval $(0, 1/M_{\lambda})$.  Now let $s = \frac{1}{M_{\lambda}} + \epsilon$.  Consider the element $\Lambda$ with $\lambda_{n-1,n-1} = \lambda_{n}$ and all other $\lambda_{ij} > \lambda_{n}$.  Inspecting Equation (\ref{finqsig}), we have the term $[k]$ with $k = \nu_{n-1,n-1} - \nu_{1} = \lambda_{1} - \lambda_{n} + (n-2)$, so
\begin{equation*}
1< ks < 2,
\end{equation*} 
since $\frac{1}{\lambda_{1} - \lambda_{n} + (n-2)} < s < \frac{2}{\lambda_{1} - \lambda_{n} + (n-2)}$.  Thus $(-1)^{\lfloor ks\rfloor} = (-1)$.  One can check that all other terms $[k']$ that appear in Equation (\ref{finqsig}) satisfy $k' < \lambda_{1} - \lambda_{n} + (n-2)$, and since $s < \frac{1}{\lambda_{1}-\lambda_{n} + (n-3)}$, we have $0 < sk' < 1$.  Thus $(-1)^{\lfloor 0 \rfloor} = 1$.  Thus $s_{q}(\zeta_{\Lambda}) = -1$, so $L(\lambda)$ is not unitary.

Now we will show that if $L(\lambda)$ is unitary, we must have the above bounds on $s$.  We will induct on $n$.  For the base case $n=3$, let $\lambda = (\lambda_{1}, \lambda_{2}, \lambda_{3})$ and assume $L(\lambda)$ is unitary.  Then the restriction of $L(\lambda)$ as a $U_{q}(\mathfrak{gl}_{2})$-module must decompose as a sum of unitary pieces, and these are indexed by $\mu$ interlacing $\lambda$.  Thus $L(\mu')$ with $\mu' = (\lambda_{1}, \lambda_{3})$ is unitary.  Similarly $L(\mu)$ with $\mu = (\lambda_{1} - 1, \lambda_{3})$ is also unitary.  By the previous lemma, this implies
$s < \frac{1}{\lambda_{1} - \lambda_{3}}$.  But in the first part of the proof, we proved that $L(\lambda)$ is not unitary for $s = \frac{1}{\lambda_{1} - \lambda_{3} + 1} + \epsilon$ for $\epsilon$ small, thus it is not unitary on $\frac{1}{\lambda_{1} - \lambda_{3} + 1} < s < \frac{1}{\lambda_{1} - \lambda_{3}}$.  Thus, we must have $0 < s < \frac{1}{\lambda_{1} - \lambda_{3} + 1}$, which establishes the base case.  We assume the result for $n-1 \geq 3$, and will show it holds for $n$.  Let $\lambda = (\lambda_{1}, \dots, \lambda_{n})$ and suppose $L(\lambda)$ is unitary.  As in the argument for $n=3$, the restriction of $L(\lambda)$ over $U_{q}(\mathfrak{gl}_{n-1})$ must decompose as a sum of unitary pieces, indexed by $\mu$ interlacing $\lambda$.  In particular, we may take $\mu$ with $\mu_{1} = \lambda_{1}$ and $\mu_{n-1} = \lambda_{n}$.  Thus, by the induction hypothesis, we must have $s < \frac{1}{\lambda_{1} - \lambda_{n} + (n-3)}$.  But in the first part of the proof, we showed that $L(\lambda)$ is not unitary for $s = \frac{1}{\lambda_{1} - \lambda_{n} + (n-2)} + \epsilon$.  Thus $0 < s < \frac{1}{\lambda_{1} - \lambda_{n} + (n-2)}$ as desired.

\end{proof}

\begin{remark}
Note that the previous result implies that $L(\lambda)$ is unitary in the classical limit $q=1$, which is known from representation theory of $\mathfrak{gl}_{n}$.  (This is also immediately seen by inspection of Equation (\ref{finqsig}) since $lim_{q \rightarrow 1} [k] = k$ and in that formula all appearing $k$ are positive.)
\end{remark}

We note that at $n=2$, there are several intervals of $s$ for which $L(\lambda)$ is unitary.  We use Lemma \ref{n=2char} to compute some examples.  Let $\lambda = (\lambda_{1}, \lambda_{2})$.  We obtain the following intervals for $s$ depending on $\lambda_{1} - \lambda_{2}$ where $L(\lambda)$ is unitary:
\begin{itemize}
\item For $\lambda_{1} - \lambda_{2} = 3$, we get $0<s<\frac{1}{3}$ and $\frac{2}{3} < s < 1$.
\item For $\lambda_{1} - \lambda_{2} = 4$, we get $0<s<\frac{1}{4}$ and $\frac{1}{2} < s < \frac{2}{3}$.
\item For $\lambda_{1} - \lambda_{2} = 5$, we get $0<s< \frac{1}{5}$ and $\frac{4}{5} < s < 1$.
\item For $\lambda_{1} - \lambda_{2} = 6$, we get $0<s< \frac{1}{6}$ and $\frac{2}{5}<s<\frac{1}{2}$ and $\frac{2}{3} < s < \frac{3}{4}$.
\item For $\lambda_{1} - \lambda_{2} = 7$, we get $0<s<\frac{1}{7}$, $\frac{1}{3} < s < \frac{2}{5}$, $\frac{3}{5} < s < \frac{2}{3}$, and $\frac{6}{7} < s < 1$.

\end{itemize}

\section{Signature characters for Verma modules}
In this section, we compute signature characters for the modules $M(\lambda)$.  The Gelfand-Tsetlin formulae in Theorem \ref{NazTar} still hold, but the basis elements are indexed by arrays, not patterns (recall the discussion in the first section).  

\begin{definition} For $\zeta_{\Lambda} \in A(\lambda)$, let 
\begin{multline*}
\tilde s_{q}(\zeta_{\Lambda}) = \prod_{1 \leq i \leq m < n} \Big\{ \prod_{i<j \leq m} [\nu_{mj} - \nu_{mi}][\nu_{mj} - \nu_{i}] \Big\} \{ [\nu_{mi}- \nu_{i}]_{(\lambda_{i} - \lambda_{m,i}-1)}! \} \\
\times \Big \{ \prod_{i < j \leq m+1} [\nu_{m+1,j} - \nu_{i}-1 ]_{(\lambda_{i}-\lambda_{m,i}-1)}! \Big \} \\
\times  \Big \{ \prod_{i \leq j \leq m-1} [\nu_{m-1,j}-\nu_{i} ]_{(\lambda_{i} - \lambda_{m,i}-1)}! \Big \},
\end{multline*}
it is equal to $\pm 1$ depending on $\Lambda$ and $q$.
\end{definition}

\begin{theorem} \label{infsigchar} We have
\begin{equation*}
ch_{s}(M(\lambda)) = \sum_{\zeta_{\Lambda} \in A(\lambda)} \tilde s_{q}(\zeta_{\Lambda}) e^{\text{wt}(\zeta_{\Lambda})}.
\end{equation*}
\end{theorem}

Recall that $\text{wt}(\cdot)$ was defined in (\ref{wt}).

\begin{proof}
Obtained from Theorem \ref{finsigchar} and Equation (\ref{finqsig}), with the only difference being the conditions on arrays.
\end{proof}

\begin{definition}
For $\zeta_{\Lambda} \in A(\lambda)$, let
\begin{multline*}
\tilde c(\zeta_{\Lambda}) = \prod_{1 \leq i \leq m < n} \Big\{ \prod_{i<j \leq m} (\nu_{mj} - \nu_{mi})(\nu_{mj} - \nu_{i}) \Big\} \\
\times \Big \{ \prod_{i < j \leq m+1} (\nu_{m+1,j} - \nu_{i}-1 )(\nu_{m+1,j}  - \nu_{i} - 2) \cdots (\nu_{m+1,j} - \nu_{i} - 1 - (\lambda_{i} - \lambda_{mi} - 1)) \Big \} \\
\times  \Big \{ \prod_{i < j \leq m-1} (\nu_{m-1,j}-\nu_{i} )(\nu_{m-1,j} - \nu_{i} - 1) \cdots (\nu_{m-1,j}-\nu_{i} - (\lambda_{i} - \lambda_{mi} - 1))\Big \}.
\end{multline*}
\end{definition}

\begin{proposition}
For $\zeta_{\Lambda} \in A(\lambda)$, we have
\begin{equation*}
\lim_{q \rightarrow 1} \tilde s_{q}(\zeta_{\Lambda}) = \tilde c(\zeta_{\Lambda}).
\end{equation*}
\end{proposition}
\begin{proof}
Follows since
\begin{equation*}
(\nu_{mi} - \nu_{i})(\nu_{mi} - \nu_{i} - 1) \cdots (\nu_{mi} - \nu_{i} - (\lambda_{i} - \lambda_{mi} - 1))
\end{equation*}
 is positive, as is
 \begin{equation*}
 (\nu_{m-1, i} - \nu_{i}) (\nu_{m-1,i} - \nu_{i} - 1) \cdots (\nu_{m-1,i} - \nu_{i} - (\lambda_{i} - \lambda_{mi} - 1)).
 \end{equation*}
 \end{proof}
 
 We use this to provide a formula for the signature character of the representation $M(\lambda)$ of $\mathfrak{gl}_{n}$.
 \begin{corollary} \label{infsigcharq=1}
 We have
 \begin{equation*}
 \lim_{q \rightarrow 1} ch_{s}(M(\lambda)) = \sum_{\zeta_{\Lambda} \in A(\lambda)} \tilde c(\zeta_{\Lambda}) e^{\text{wt}(\zeta_{\Lambda})}.
 \end{equation*}
 \end{corollary}

We now study the case where $\lambda$ is in a particular region, termed the \textit{Wallach region} (as in \cite{WY}):
\begin{equation*}
\lambda_{1} < \lambda_{2} < \cdots < \lambda_{n}
\end{equation*}
with $\lambda_{i+1} - \lambda_{i} \geq 1$.   In this case, the signature character has an explicit factorized form.  We note that this was first computed by Wallach (for all Lie algebras, not just $\mathfrak{gl}_{n}$), and then used by Yee in determining signature characters for $\lambda$ in any region \cite{NW, WY, WY2}.

\begin{proposition} \label{Wall}
For $\lambda$ in the Wallach region, we have
\begin{equation*}
\tilde c(\zeta_{\Lambda}) = \prod_{1 \leq i \leq n-1} (-1)^{\lambda_{i} - \lambda_{ii}}.
\end{equation*}
\end{proposition}
\begin{proof}
We use the formula (from the previous section, although one can use either)
\begin{multline*}
\tilde s_{q}(\zeta_{\Lambda}) = \prod_{1 \leq i \leq m < n} \Big\{ \prod_{i<j \leq m} [\nu_{mi}^{\Lambda} - \nu_{mj}^{\Lambda}][\nu_{i} - \nu_{mj}^{\Lambda}] \Big\}  (-1)^{\nu_{mi}^{\Lambda} - \nu_{i}} \{ [\nu_{mi}^{\Lambda} - \nu_{i}]_{(\lambda_{i} - \lambda_{m,i}-1)}! \} \\
\times \Big \{ \prod_{i < j \leq m+1} [\nu_{mi}^{\Lambda} - \nu_{m+1,j}^{\Lambda}]_{(\lambda_{i}-\lambda_{m,i}-1)}! \Big \} \Big \{ \prod_{i \leq j \leq m-1} [\nu_{mi}^{\Lambda} - \nu_{m-1,j}^{\Lambda} - 1]_{(\lambda_{i} - \lambda_{m,i}-1)}! \Big \}. \\
\end{multline*}
First note that the signature of an arbitary weight space is constant over this region, since the region does not intersect any degeneracy hyperplanes.  So to compute the signature of a fixed weight space, we can take the differences $\lambda_{i+1} - \lambda_{i}$ to be very large.  We check for fixed $1 \leq i \leq m < n$:
\begin{equation*}
\nu_{mi} - \nu_{mj} > 0, \nu_{i} - \nu_{mj} > 0
\end{equation*}
\begin{equation*}
\nu_{mi} - \nu_{i} > 0
\end{equation*}
\begin{equation*}
\nu_{mi} - \nu_{m+1, j} > 0
\end{equation*}
\begin{equation*}
\nu_{mi} - \nu_{m-1,j} - 1 > 0, \text{ unless } j=i \text{ in which case } <0.
\end{equation*}
So for $i \neq m$, we get sign $(-1)^{\nu_{mi} - \nu_{i}}(-1)^{\lambda_{i} - \lambda_{mi}} = +1$ and for $i = m$, we get $(-1)^{\nu_{ii} - \nu_{i}}$.  So the total is
\begin{equation*}
\prod_{1 \leq i \leq n-1} (-1)^{\lambda_{i} - \lambda_{ii}},
\end{equation*}
as desired.

Note that one could instead used Lemma \ref{onestep} (i.e., just increasing the node $(m,i)$ by one) and iterate to prove the result.  The sign of that step is $-1$ if $i=m$ and $1$ otherwise.
\end{proof}

\begin{corollary}
The formula for the signature character in the Wallach region is 
\begin{equation*}
 \frac{e^{\lambda}}{\displaystyle \prod_{1 \leq i \leq m<n} \Big( 1 + \frac{x_{m}}{x_{i}} \Big)}.
\end{equation*}
\end{corollary}
\begin{proof}
We consider moves corresponding to $(m,i)$ that increase entries in $(r, i)$ for $r \leq m$ by one unit.  By Proposition \ref{Wall}, such a move induces a sign change of $(-1)$.  We also have
\begin{equation*}
\frac{1}{1 + \frac{x_{m}}{x_{i}}} = 1 - (x_{m}/x_{i}) + (x_{m}/x_{i})^{2} - (x_{m}/x_{i})^{3} + \cdots.
\end{equation*}
On the other hand, such moves (varying over all $1 \leq i \leq m<n$) generate all possible weights.  This provides the desired bijection between the two quantities.
\end{proof}

\section{Examples}

In this section, we use Theorem \ref{finsigchar} and Theorem \ref{infsigchar} to compute some examples of signature characters for small values of $n$.  We will also use Corollary \ref{infsigcharq=1} to compute signature characters of Verma modules in the classical limit $q \rightarrow 1$ for small values of $n$.

\begin{enumerate}

\item \textbf{$n=2$, formula for $ch_{s}(L(\lambda))$.}  We let $\lambda = (\lambda_{1}, \lambda_{2})$ with $\lambda_{1} - \lambda_{2} \in \mathbb{Z}_{+}$.  We index Gelfand-Tsetlin patterns by $\Lambda_{i}$ with $\lambda_{1}, \lambda_{2}$ in the first row and $\lambda_{11} = \lambda_{1} - i$ in the second row for $i=0,1, \dots, (\lambda_{1} - \lambda_{2})$.  We use Definition \ref{eltsgn} to compute
\begin{multline*}
s(\zeta_{\Lambda_{i}}) = (-1)^{\nu_{11} - \nu_{1}} \{ [\nu_{11} - \nu_{1}]_{(\lambda_{1} - \lambda_{11} - 1)}! \} \{ [\nu_{11} - \nu_{22}]_{(\lambda_{1} - \lambda_{11}-1)}! \} \\
= (-1)^{i} \{ [i]_{(i-1)}! \} \{ [ (\lambda_{2}- \lambda_{1}) +i - 1]_{(i-1)}! \}.
\end{multline*}
So we have
\begin{equation*}
ch_{s}(L(\lambda)) = \sum_{0 \leq i \leq (\lambda_{1} - \lambda_{2})}  \{ [i]! \} \{ [ \lambda_{1} - \lambda_{2}]_{(i-1)}! \} e^{( \lambda_{1} - i,  \lambda_{2} + i)}.
\end{equation*}

\item \textbf{$n = 3$, formula for $ch_{s}(L(\lambda))$.}  We let $\lambda = (\lambda_{1}, \lambda_{2}, \lambda_{3})$, with $\lambda_{i} - \lambda_{i+1} \in \mathbb{Z}_{\geq 0}$.  We let $\Lambda = (\lambda_{1}, \lambda_{2}, \lambda_{3}; \lambda_{1} - d_{1}, \lambda_{2} - d_{2}; \lambda_{1} - d_{1} - d_{3})$ denote an arbitrary Gelfand-Tsetlin pattern indexing the basis.  Note that we must have $0 \leq d_{1} \leq (\lambda_{1} - \lambda_{2})$ and $0 \leq d_{2} \leq (\lambda_{2} - \lambda_{3})$ and $0 \leq d_{3} \leq (\lambda_{1} - d_{1}) - (\lambda_{2} - d_{2})$ to satisfy the interlacing condition.  We use Equation (\ref{finqsig}) to compute
\begin{multline}\label{n=3sgn}
s_{q}(d_{1}, d_{2}, d_{3}) = s_{q}(\zeta_{\Lambda})  \\ = \{ [d_{1}]! [d_{2}]! [d_{1} + d_{3}]! [d_{2} - d_{1} + \nu_{2} - \nu_{1}]  [d_{2} + \nu_{2}-\nu_{1}] \\ \times [\nu_{2}-\nu_{1}-1]_{(d_{1}-1)}! [\nu_{3}-\nu_{1}-1]_{(d_{1}-1)}! \} 
 \{ [d_{1} + d_{3}]_{(d_{1}-1)}! [\nu_{3}-\nu_{2}-1]_{(d_{2}-1)}! [d_{2}+\nu_{2}-\nu_{1}]_{(d_{1}+d_{3}-1)}!\}.
\end{multline}
So we have
\begin{equation*}
\sum_{\substack{0 \leq d_{1} \leq (\lambda_{1} - \lambda_{2}) \\ 0 \leq d_{2} \leq (\lambda_{2} - \lambda_{3}) \\ 0 \leq d_{3} \leq (\lambda_{1} - d_{1})- (\lambda_{2} - d_{2})}} s_{q}(d_{1},d_{2},d_{3}) e^{(\lambda_{1} - d_{1} - d_{3}, \lambda_{2} - d_{2} + d_{3}, \lambda_{3} + d_{1} + d_{2})},
\end{equation*}
where $s_{q}(d_{1},d_{2},d_{3})$ is as computed in (\ref{n=3sgn}).

\item \textbf{$n=2$, formula for $ch_{s}(M(\lambda))$.}  We let $\lambda = (\lambda_{1}, \lambda_{2})$ with $\lambda_{1} - \lambda_{2} \in \mathbb{R}$.  We have the same formula as $ch_{s}(L(\lambda))$, except the basis is parametrized differently:
\begin{equation*}
ch_{s}(M(\lambda)) = \sum_{0 \leq i} (-1)^{i} \{[i]_{i-1}!\} \{ [(\lambda_{2} - \lambda_{1}) + i-1 ]_{i-1}!\} e^{(\lambda_{1} - i, \lambda_{2} + i)}.
\end{equation*}

\item \textbf{$n=3$, formula for $ch_{s}(M(\lambda))$.}  We let $\lambda = (\lambda_{1}, \lambda_{2}, \lambda_{3})$ with $\lambda_{i} - \lambda_{i+1} \in \mathbb{R}$.  We have the same formula as $ch_{s}(L(\lambda))$, except the basis is parametrized differently:
\begin{equation*}
ch_{s}(M(\lambda)) = \sum_{d_{1}, d_{2}, d_{3} \geq 0} s_{q}(d_{1}, d_{2}, d_{3}) e^{(\lambda_{1} - d_{1} - d_{3}, \lambda_{2} - d_{2} + d_{3}, \lambda_{3} + d_{1} + d_{2})}, 
\end{equation*}
where $s_{q}(d_{1}, d_{2}, d_{3})$ is computed in (\ref{n=3sgn}).

\item \textbf{$n=2$, formula for $\lim_{q \rightarrow 1} ch_{s}(M(\lambda))$.}  Using (3) and Equation (\ref{sgncomp}), we have
\begin{multline*}
\lim_{q \rightarrow 1} ch_{s}(M(\lambda)) \\ = \sum_{0 \leq i} (-1)^{i} \{ (i)(i-1) \cdots (1) \} \{ (\lambda_{2} - \lambda_{1} + i - 1)(\lambda_{2} - \lambda_{1} + i -2) \cdots (\lambda_{2} - \lambda_{1}) \} e^{(\lambda_{1}-i, \lambda_{2}+i)} \\
= \sum_{0 \leq i} (-1)^{i} (-1)^{\min \{0, \lfloor \lambda_{2} - \lambda_{1} + i - 1 \rfloor + 1\}} (-1)^{\min\{0, \lfloor\lambda_{2} - \lambda_{1} \rfloor  \}} e^{(\lambda_{1}-i, \lambda_{2}+i)}.
\end{multline*}

\item \textbf{$n=3$, formula for $\lim_{q \rightarrow 1} ch_{s}(M(\lambda))$.}
We note that in (\ref{n=3sgn}), the terms $d_{1}, d_{2}, d_{1} + d_{3}$ are all positive, so we have
\begin{multline*}
\lim_{q \rightarrow 1} s_{q}(d_{1}, d_{2}, d_{3}) = \{ (d_{2} - d_{1} + \nu_{2} - \nu_{1})(d_{2} + \nu_{2} - \nu_{1})(\nu_{2} - \nu_{1} - 1)_{(d_{1} - 1)}! \\ \times (\nu_{3} - \nu_{1}-1)_{(d_{1} - 1)}! (\nu_{3}-\nu_{2}-1)_{(d_{2}-1)}! (d_{2} + \nu_{2} - \nu_{1})_{(d_{1} + d_{3}-1)}! \}.
\end{multline*}
Thus,
\begin{multline*}
\lim_{q \rightarrow 1} ch_{s}(M(\lambda)) =(-1)^{\min \{ 0, \lfloor \nu_{2} - \nu_{1} -1 \rfloor +1 \}} (-1)^{\min\{0, \lfloor \nu_{3}-\nu_{2}-1 \rfloor +1 \}}  (-1)^{\min \{ 0, \lfloor \nu_{3} - \nu_{1} - 1 \rfloor + 1\}} \\ \times \sum_{d_{1}, d_{2}, d_{3} \geq 0} \Bigg( \{d_{2} - d_{1} + \nu_{2} - \nu_{1} \} \{d_{2} + \nu_{2} - \nu_{1} \}   (-1)^{\min \{ 0, \lfloor \nu_{2} - \nu_{1} - 1 - (d_{1}-1) \rfloor \}} \\ \times  (-1)^{\min\{ 0,\lfloor \nu_{3} - \nu_{1} - 1 - (d_{1} - 1) \rfloor \} }  (-1)^{\min \{0, \lfloor \nu_{3}-\nu_{2}-1-(d_{2}-1) \rfloor \}}  (-1)^{\min\{0, \lfloor d_{2} + \nu_{2}-\nu_{1} \rfloor + 1 \}}  \\ \times (-1)^{\min\{0, \lfloor d_{2} + \nu_{2}-\nu_{1} - (d_{1} +d_{3}-1) \rfloor \}} \Bigg)   e^{(\lambda_{1} - d_{1} - d_{3}, \lambda_{2} - d_{2} + d_{3}, \lambda_{3} + d_{1} + d_{2})};
\end{multline*}
recall that $\nu_{i} = i-\lambda_{i} - 1$.

\end{enumerate}

\end{document}